\documentclass[leqno, 10.5pt]{amsart}

\usepackage{amsfonts,amsmath,amssymb,amsthm,amscd,enumerate,comment,gensymb}
\usepackage[dvipsnames]{xcolor}
\usepackage{hyperref}

\newtheorem{thm}{Theorem}[section]

\newtheorem{lem}{Lemma}[section]

\theoremstyle{remark}
\newtheorem{rmk}{Remark}[section]

\numberwithin{equation}{section}

\theoremstyle{definition}

\newcommand{\N}{\ensuremath{\mathbb{N}}}

\newcommand{\R}{\ensuremath{\mathbb{R}}}

\newcommand{\rS}{\ensuremath{\mathbb{S}}}

\newcommand{\mC}{\mathcal{C}}
\newcommand{\mL}{\mathcal{L}}
\newcommand{\na}{\nabla}
\newcommand{\la}{\langle}
\newcommand{\ra}{\rangle}
\newcommand{\lla}{\left \langle}
\newcommand{\rra}{\right \rangle}
\newcommand{\lp}{\left(}
\newcommand{\pa}{\partial}
\newcommand{\rp}{\right)}
\newcommand{\vphi}{\varphi}

\newcommand{\graph}{\textup{graph}}

\newcommand{\lan}{\left \langle}
\newcommand{\ran}{\right \rangle}

\newcommand{\lab}{\left|}
\newcommand{\rab}{\right|}
\newcommand{\Lab}[1]{\left| #1 \right|}

\DeclareMathOperator{\ee}{e}

\numberwithin{equation}{section}

\begin{document}

\title[Convexity of 2-convex translating and expanding solitons]{Convexity of 2-convex translating and expanding solitons to the mean curvature flow in $\R^{n+1}$}
\author{Junming Xie and Jiangtao Yu}

\address{Department of Mathematics, Lehigh University, Bethlehem, PA 18015}
\email{jux216@lehigh.edu}

\address{Institute of Geometry and Physics, University of Science and Technology of China, Hefei, Anhui, China 230026}
\email{jtao\_yu@126.com}

\begin{abstract}
	In this paper, inspired by the work of Spruck-Xiao \cite{SX:20} and based partly on a result of Derdzi\'nski \cite{De:80}, we prove the convexity of complete 2-convex translating and expanding solitons to the mean curvature flow in $\R^{n+1}$. More precisely, for $n\geq 3$, we show that any $n$-dimensional complete 2-convex translating solitons are convex, and any $n$-dimensional complete 2-convex self-expanders asymptotic to (strictly) mean convex cones are convex.
\end{abstract}
\maketitle

\section{Introduction}

{\it Self-similar solutions} of the mean curvature flow are special solutions for which the shapes of the evolving surfaces are essentially preserved. They play a very important role in the singularity analysis of the mean curvature flow. In particular, self-shrinkers and translating solitons often arise as possible finite time Type I and Type II singularity models, respectively (see, e.g., \cite{Huisken:90}, \cite{Angenent:91,HS:99}). On the other hand, self-expanders are expected to model the long time behavior of the mean curvature flow or the behavior of the flow as it emerges from conical singularities (see, e.g., \cite{EH:89,ACI:95,Stavrou:98,Rasul:10}). In this paper, we focus on translating solitons and self-expanders with codimension one in $\R^{n+1}$ and study their convexity under the 2-convex assumption.

A {\it translating soliton} (or a {\it translator} for short) to the mean curvature flow is an immersed hypersurface $\Sigma^n$ in $\R^{n+1}$ 
satisfying the equation
\begin{equation} \label{eq:translator}
	H=\la \ee_{n+1},\nu \ra.
\end{equation} 
Here and in the following, $H$ is the mean curvature, $\nu$ denotes a global unit normal vector field and $\ee_{n+1}$ is a unit direction. On the other hand, an immersed hypersurface $\Sigma^n$ in $\R^{n+1}$ is called a {\it self-expander} if there is an immersion $x: \Sigma^n \rightarrow \R^{n+1}$ such that
\begin{equation} \label{eq:selfexpander}
	H =\la x, \nu \ra.
\end{equation}

For translating solitons, when $n=1$, the unique convex solution is the grim reaper curve $\Gamma$, which is the graph of $y = \log \sec{x},\ \lab x \rab < \pi/2$. For $n \ge 2$, $\Sigma = \Gamma \times \R^{n-1}$ are convex translating solitons, where $\Gamma$ is the $1$-dimensional grim reaper. There is an entire graphical, rotationally symmetrical and strictly convex translating soliton called the ``bowl soliton'' (see \cite{CSS:07}) in $\R^{n + 1}\ (n\ge2)$. Note that the bowl soliton has the asymptotic expansion as an entire graph $u(x) = \frac{1}{2(n-1)}\lab x\rab^2 - \log\lab x \rab  + O(1)$.

There have been a lot of efforts to classify convex translating solitons. X.-J. Wang \cite{Wang:11} proved that, for all dimensions $n\geq 1$, a complete convex translating soliton is either an entire solution or is defined in a strip region. He also showed that the only entire convex graphical translating soliton in $\R^3$ is the bowl soliton. As for solutions over slab regions, Bourni-Langford-Tinaglia \cite{BLT:20b} proved that there exists a strictly convex translating soliton in a slab region with width greater than $\pi$ in $\R^{n + 1}\ (n\geq 2)$. They solved the Dirichlet problem to obtain such a translating soliton. Meanwhile, Hoffman-Ilmanen-Mart\'in-White \cite{HIMW:19} further proved the existence for all slabs of width greater than $\pi$ for the case of dimension $n=2$. Moreover, the same authors \cite{HIMW:19}  are able to prove the uniqueness in this case.

Regarding the pinching estimates for translating solitons, when $n=2$, Spruck and Xiao \cite{SX:20} recently proved that any immersed two-sided mean convex translating soliton in $\R^3$ is convex. Here and in the following, {\it mean convex} means the mean curvature $H>0$. As an application of Spruck-Xiao's convexity theorem, Hoffman-Ilmanen-Mart\'in-White \cite{HIMW:19} gave a full classification of the translating solitons that are complete graphs over domains in $\R^2$. Later on, Bourni-Langford-Tinaglia \cite{BLT:20b} extended Spruck-Xiao's result by showing that a mean convex translating soliton in $\R^{n+1}$ with at most two distinct principal curvatures at each point and bounded second fundamental form is convex.

More recently, Spruck and Sun \cite{SS:20} proved the convexity of  any immersed two-sided translating soliton $\Sigma\subset \R^{n+1}\ (n\ge3)$ under the assumption that $\Sigma$ is uniformly 2-convex. Here, the {\it uniformly 2-convex} condition means that $H>0$ and $\kappa_1 + \kappa_2\ge \beta H$ for some constant $\beta>0$, where $\kappa_1\leq \kappa_2$  are the least two principal curvatures. Although $\kappa_1$ is smooth wherever $\kappa_1<0$, the authors were concerned that the other principal curvatures might not be differentiable if some multiplicity changes. In \cite{SS:20}, they successfully dealt with this issue by using a special
$\delta$-approximation of $\min\{\kappa_1,\ldots, \kappa_n\}$.
They also needed the uniformly 2-convex condition to show that the $\delta$-approximation $\mu^{n} = \mu^n(\kappa_1,\cdots,\kappa_n)$ is smooth and converges to $\min\{\kappa_1,\ldots, \kappa_n\}$. As pointed out in \cite{SS:20}, their argument in applying the Omori-Yau maximum principle is delicate and utilizes the special properties of their approximation $\mu^n$.

It turns out, by applying a result of Derdzi\'nski \cite{De:80}, we can actually avoid the elaborate $\delta$-approximation arguments in~\cite{SS:20} and work with the eigenfunctions $\kappa_1,\kappa_2,\cdots,\kappa_n$ directly. Our first main result is the following.

\begin{thm}{\label{thm:translator}}
Let $\Sigma^n \subset \R^{n+1}$, $n\geq 3$, be a complete immersed two-sided 2-convex translating soliton for the mean curvature flow. Then $\Sigma$ is convex.
\end{thm}

\begin{rmk}
Most computations in the proof of Theorem \ref{thm:translator} was carried out in the Fall of 2018 after we learned the result of Spruck-Xiao \cite{SX:20} through the lecture of B. White at the Lehigh Geometry and Topology Conference in May 2018. In November 2020, the work of Spruck-Sun \cite{SS:20} was brought to our attention. We immediately noticed their concern about the differentiability of the principal curvatures and realized that we had overlooked this issue. However, as the second fundamental form is a Codazzi tensor, we can get around this potential issue by appealing to a result of Derzi\'nski \cite{De:80} on the set of $\{\kappa_1<0\}$ (see Lemma \ref{lma:derzinski}).
\label{rmk:mainThm}
\end{rmk}

For self-expanders, Smoczyk \cite{Smoczyk:21} proved that any properly immersed mean convex self-expanding surface $\Sigma^2 \subset \R^3$ that is smoothly asymptotic to a cone must be strictly convex. Partly motivated by this work, we also study the pinching estimate for higher dimensional 2-convex asymptotically conical self-expanders. Specifically, by using similar arguments as in the proof of Theorem \ref{thm:translator}, we prove the following result.

\begin{thm} \label{thm:selfexpander}
	Let $\Sigma^n \subset \R^{n+1}$, $n\geq 3$, be a complete immersed two-sided 2-convex self-expander asymptotic to mean convex cones. Then $\Sigma^n$ is convex.
\end{thm}

\begin{rmk}
	By essentially the same argument, we can also prove a similar result for the self-shrinking case. Namely, we are able to prove that any complete immersed two-sided 2-convex self-shrinker asymptotic to mean convex cones must be convex. See Theorem \ref{thm:selfshrinker}.
\end{rmk}

\begin{rmk}
	Actually, by the assumption of 2-convexity, the asymptotic cone must be a (weakly) convex cone (see the proof of Theorem \ref{thm:selfexpander} in section 4). However, we need the assumption of the asymptotic cone being mean convex to ensure that the asymptotic cone has at least one positive principal curvature.
\end{rmk}

We also note that it is natural to assume being asymptotically conical for self-expanders. On one hand, a lot of known examples of self-expanders are asymptotically conical (see, e.g., \cite{EH:89}, \cite{Ilmanen:95,Ding:20}, \cite{ACI:95,Helmensdorfer:12}). In particular, the mean curvature flow associated to the asymptotically conical self-expander flows out of the asymptotic cone. On the other hand, in recent years, some advances in self-expanders were obtained under the assumption of being asymptotically conical. For example, regarding the uniqueness of self-expanders, Fong-McGrath \cite{FongM:19} proved that mean convex self-expanders which are asymptotic to $O(n)$-invariant cones are rotationally symmetric. Ding \cite{Ding:20} has studied the relationship between minimal cones and self-expanders. More recently, Bernstein and Wang \cite{BW:19,BW:21} studied the space of asymptotically conical self-expanders of mean curvature flow, in particular compactness in the locally smooth topology for certain natural families of asymptotically conical self-expanding solutions of mean curvature flow.

\medskip
\textbf{Acknowledgements.} Both authors are very grateful to their Ph.D. advisor Professor Huai-Dong Cao for suggesting this project in Summer 2018 and for his invaluable suggestions, constant support and encouragement. The first author would also like to thank Professor Lu Wang for helpful discussions, and Professor Knut Smoczyk for answering questions related to his paper \cite{Smoczyk:21}.

\bigskip
\section{Preliminaries}

In this section, we fix the notations for the rest of the paper and collect several results that will be used later in the proof of the main theorems.

\subsection{Codazzi tensors}

A symmetric $(0,2)$ tensor $T$ on a Riemannian manifold is said to be a {\em Codazzi tensor} if it satisfies the {\em Codazzi equation}
$$ (\na_XT)(Y,Z) = (\na_YT)(X,Z) $$
for any tangent vector fields $X, Y$, and $Z$. Now, we recall a very useful result of Derdzi\'nski \cite{De:80} (valid for any Codazzi tensor) that we shall need in the proof of our main theorems. Define the integer-valued function $E_A$ for the second fundamental form $A$ on $\Sigma$ by
\[E_A(p) = \text{the number of distinct eigenvalues of }A(p).\]
Let $\Sigma_A = \{p\in\Sigma: E_A \text{ is constant in a neighborhood of }p\}$, then $\Sigma_A$ is open and dense in $\Sigma$ (see~\cite{De:80}).
We now set \[\Sigma^-=\{p\in\Sigma: \text{the least eigenvalue of } A \text{ is negative at $p$}\}\]
and  define \[\Sigma_A^- := \Sigma^- \cap \Sigma_A.\]
It is clear that $\Sigma_A^-$ is an open dense subset of $\Sigma^-$. Notice also the 2-convexity condition of $\kappa_1+\kappa_2>0$ ensures that $\kappa_1$ is smooth in the open set $\Sigma^-$. 

\begin{lem} \footnote{In early April 2021, we became aware of the work of Singley \cite{Singley:75} and found that his work suits our needs more directly.} \textup{{\bf (Derdzi\'nski \cite{De:80})}}
	Given any Codazzi tensor $A$ on a Riemannian manifold $\Sigma^n$, we have, on each connected component $\Omega$ of $\Sigma_A^-$,
	\begin{enumerate}
		\item The eigendistributions of $A$ are integrable and their leaves are
		totally umbilic submanifolds of $\Sigma^-$,
		\item Eigenspaces of $A$ form mutually orthogonal differentiable distributions.
	\end{enumerate}
	\label{lma:derzinski}
\end{lem}

Let $E_1, \cdots, E_m$ be the local eigendistributions of $A$ around any point $p\in \Omega$. Let the dimension of the $E_l$ be $d_l$, $l=1, \cdots, m$. Then there exist local orthonormal eigenvector fields $\{\tau_i\}_{i=1}^n$ around $p$ such that 
\[ \{\tau_1,\cdots, \tau_{d_1}\} \in E_1, \ \{\tau_{d_1 + 1},\cdots, \tau_{d_1+d_2}\}\in E_2, \cdots, 
 \{\tau_{d_1+\cdots+d_{m-1}+1},\cdots,\tau_n\} \in E_m.\] 
Define \[\kappa_i = A(\tau_i,\tau_i).\]
Then, with suitable labeling, locally we can have $\kappa_1 \leq  \kappa_2 \leq \cdots \leq \kappa_n$.

From now on, we shall call the above special local (tangent) frame $\{\tau_i\}_{i=1}^{n}$ the {\it adapted moving frame} around point $p$.

\begin{lem}
	On each connected component $\Omega$ of $\Sigma_A^-$, let $\{\tau_j\}_{j=1}^{n}$ be the adapted moving frame around a point $p$ in $\Omega$. Then, for any eigenvalue $\kappa_i$ of multiplicity one and at the point $p$, we have
\begin{equation}
	{\nabla{\kappa}_i}={(\nabla{A})(\tau_i,\tau_i)},
\end{equation}
\begin{equation}
	{\Delta{\kappa}_i} ={(\Delta{A})(\tau_i,\tau_i)}-2\sum_{l \neq i}\frac{\Lab{{(\nabla{A})(\tau_l,\tau_i)}}^2}{\kappa_l-\kappa_i}.
\end{equation}
	\label{lem:laplacian_k1}
\end{lem}

\begin{proof}
	We shall use the convention  that $\nabla_j$ means $\nabla_{\tau_j}$. With the adapted moving frame $\{\tau_j\}_{j=1}^{n}$, we have
	\begin{equation}
		\left\langle \tau_i,\, \tau_i \right\rangle=1\quad \implies\quad \left\langle\nabla_j\tau_i, \tau_i\right\rangle=0\quad \forall\ j =1, \cdots, n.
		\label{equ:tii}
	\end{equation}
	Then, the first conclusion (2.1) follows easily:  for any local vector field $\tau$ around $p$, 
	\begin{equation}
		{\nabla_{\tau}{\kappa}_i}={\nabla_{\tau} (A(\tau_i,\tau_i))}=(\nabla_{\tau} A)(\tau_i,\tau_i)+2A(\nabla_{\tau} \tau_i,\tau_i)
		=(\nabla_{\tau} A)(\tau_i,\tau_i).
		\label{equ:eigen_ii}
	\end{equation}
	
	\smallskip
	Next, for any $l\ne i$, we have
	\begin{equation}
		\langle \tau_l,\, \tau_i \rangle=0\quad \implies\quad \left\langle\nabla_j\tau_l,\,\tau_i\right\rangle+\left\langle\tau_l,\,\nabla_j\tau_i\right\rangle=0 \quad \forall\ j.
		\label{equ:tij1}
	\end{equation}
	We can differentiate (\ref{equ:tii}) once more  to get
	\begin{equation}
		\langle \nabla_j\tau_i,\, \tau_i \rangle=0\quad \implies \quad
		\langle \nabla_j\nabla_j\tau_i,\, \tau_i \rangle+\langle \nabla_j\tau_i,\, \nabla_j\tau_i \rangle=0.
		\label{equ:tij2}
	\end{equation}

	From $A(\tau_l,\tau_i)=0$ for any  $l\ne i$, we get ${\nabla_j (A(\tau_l,\tau_i))}=0$, which implies that
	\begin{equation}
		\begin{aligned}
	(\nabla_j A)(\tau_l,\tau_i)&=-A(\nabla_j \tau_l,\tau_i)-A(\tau_l,\nabla_j \tau_i)\\
	&=-\kappa_i\left\langle\nabla_j \tau_l,\tau_i\right\rangle-\kappa_l\left\langle\tau_l,\nabla_j \tau_i\right\rangle\\
	&=-(\kappa_l-\kappa_i)\left\langle\tau_l,\nabla_j \tau_i\right\rangle,
		\end{aligned}
		\label{equ:eigen_li}
	\end{equation}
	where we have used  (\ref{equ:tij1}) in the last equality.

	Let $\xi_i=\tau_i(p)$ and extend $\xi_i$ by parallel translation
	along radial geodesics emanating from $p$. Then, at $p$, we have
	\begin{equation}
		\begin{aligned}
			{\Delta{\kappa}_i}&=\sum_j{\nabla_{\xi_j}\nabla_{\xi_j}(A(\tau_i,\tau_i))}\\
			&=\sum_j{\nabla_{\xi_j}((\nabla_{\xi_j} A)(\tau_i,\tau_i))}\\
			&={(\Delta A)(\tau_i,\tau_i)}+2\sum_j(\nabla_{\xi_j} A)(\nabla_{\xi_j}\tau_i,\tau_i)\\
			&={(\Delta A)(\tau_i,\tau_i)}+2\sum_j(\nabla_{j} A)(\nabla_{j}\tau_i,\tau_i).
		\end{aligned}
		\label{equ:eigen_ii_laplacian}
	\end{equation}
	Here,  in the second equality, we have used (\ref{equ:tii}).
	Now,  it remains to compute and analyze the term $(\nabla_j A)(\nabla_j\tau_i,\tau_i)$.  
	By using $A(\nabla_j\tau_i,\tau_i)=0$, we obtain
	\begin{equation}
		(\nabla_j A)(\nabla_j\tau_i,\tau_i)+A(\nabla_j\nabla_j\tau_i,\tau_i)+A(\nabla_j\tau_i,\nabla_j\tau_i)=0.
		\label{equ:key}
	\end{equation}
	For the second term in (\ref{equ:key}), as $\nabla_j\tau_i=\sum_{l\ne i}\left\langle\nabla_j\tau_i,\tau_l\right\rangle\tau_l$, it follows from (\ref{equ:tij2}) that
	\begin{equation*}
		A(\nabla_j\nabla_j\tau_i,\tau_i)=-\kappa_i\left\langle\nabla_j\tau_i,\nabla_j\tau_i\right\rangle=-\sum_{l\ne i}\kappa_i(\left\langle\nabla_j\tau_i,\tau_l\right\rangle)^2.
	\end{equation*}
	For the last term in (\ref{equ:key}), we have
	\begin{equation*}
		A(\nabla_j\tau_i,\nabla_j\tau_i)=\sum_{l\ne i}\kappa_l(\left\langle\nabla_j\tau_i,\tau_l\right\rangle)^2.
	\end{equation*}
	Hence, by using (\ref{equ:eigen_li}), we get
	\begin{equation}
			(\nabla_j A)(\nabla_j\tau_i,\tau_i)=-\sum_{l\ne i}(\kappa_l-\kappa_i)(\left\langle\nabla_j\tau_i,\tau_l\right\rangle)^2=-\sum_{l \neq i}\frac{\Lab{{(\nabla_j{A})(\tau_l,\tau_i)}}^2}{\kappa_l-\kappa_i}.
		\label{equ:eigen_}
	\end{equation}
	The second conclusion (2.2) in Lemma 2.3 now follows from (\ref{equ:eigen_ii_laplacian}) and (\ref{equ:eigen_}).
\end{proof}

\subsection{Translating solitons}
Recall the following well-known basic properties of translating solitons for the mean curvature flow (see, e.g., \cite{MSS:15,Xin:15,SS:20}).

\begin{lem}\label{pre:translator}
	Let $f:\Sigma^n\to\R^{n+1}$ be a complete immersed two-sided translating soliton satisfying Eq. (\ref{eq:translator}). Let $A=(h_{ij})$ be the second
	fundamental form, 
	$\Delta^f := \Delta ^\Sigma+\lan \nabla \cdot ,\  \ee_{n+1} \ran$ be the drift Laplacian on $\Sigma$. Then
	\begin{enumerate}
		\item $\Delta^f H+\Lab{A}^2H=0$,
		\item $\Delta^f A+\Lab{A}^2A=0$,
		\item $\Delta^f (\Lab{A}^2) +2\Lab{A}^4 -2\Lab{\nabla A}^2=0$.
	\end{enumerate}
\end{lem}

In the proof of Theorem \ref{thm:translator}, we shall need the following version of the Omori-Yau maximal principle (see \cite{HIMW:21}).

\begin{lem} \textup{{\bf (Omori-Yau Maximal Principle)}} \label{lemmaomoriyau}
	Let $(M, g)$ be a complete Riemannian manifold with Ricci curvature bounded from below. Let $f: M\rightarrow \R$ be a smooth function that is bounded below. Then there is a sequence $p_k$ in $M$ such that
	$$  f(p_k) \rightarrow \inf_Mf, \quad \nabla f(p_k) \rightarrow 0, \quad \liminf \Delta f(p_k) \geq 0.  $$ 
	Moreover,  the theorem remains true if we replace the assumption that $f$ is smooth by the assumption that $f$ is smooth on $\{f<\alpha\}$ for some $\alpha >\inf_Mf$.
\end{lem}

\begin{rmk} \label{OY} Observe that, by the translator equation (\ref{eq:translator}), $H\le1$. This implies that if $\Sigma^n\subset \R^{n+1}$ is 2-convex then $\left| A \right|^2\le nH^2 \le n$. On the other hand, by the Gauss equation, $R_{ij}=Hh_{ij}-h_{ik}h_{kj}$.  Thus,  for any 2-convex translator $\Sigma$, its Ricci curvature is automatically bounded from below so  Lemma 2.4 applies.
\end{rmk}

\subsection{Self-expanders}
For self-expanders of the mean curvature flow, we have the following well-known basic properties (see, e.g.,~\cite{Ding:20}).

\begin{lem} \label{pre:expander}
	Let $x: \Sigma^n \rightarrow \R^{n+1}$ be a complete immersed two-side self-expanders satisfying Eq. (\ref{eq:selfexpander}). Let $A=(h_{ij})$ be the second fundamental form, $\mL:=\Delta^{\Sigma}+\la \na \cdot,x \ra$ be the drift Laplacian on $\Sigma$. Then
	\begin{enumerate}
		\item $ \mL H + H(1+|A|^2)=0 $,
		\item $ \mL A + A(1+|A|^2) = 0 $,
		\item $ \mL (|A|^2) + 2|A|^2(1+|A|^2) - 2|\na A|^2 =0 $.
	\end{enumerate}
\end{lem}

Finally, we recall the definition of asymptotically conical self-expander and some of its geometric properties. We say that a cone $\mC\subset \R^{n+1}$ is a {\it regular cone} with vertex at the origin, if $\mC = \{\rho \Gamma^{n-1}, 0\leq \rho <\infty\}$, where $\Gamma^{n-1}$ is a smooth connected closed embedded submanifold of the unit sphere $\rS^n$. Note that the induced metric on $\mC$ satisfies $g_{\mC} = dr\otimes dr + r^2g_{\Gamma}$, for $r$ the radial variable.

A self-expander $\Sigma^n \subset \R^{n+1}$ is {\it asymptotically conical} if it is smoothly asymptotic to a regular cone $\mC$ such that
\begin{equation*}
	\mC = \lim_{t\rightarrow 0^+} \sqrt{t}\Sigma,
\end{equation*}
in the sense that, for any $R>0$ and $k\in \N$, $\sqrt{t}\Sigma\cap (\bar{B}_R\backslash B_{1/R})$ converges to $\mC\cap (\bar{B}_R\backslash B_{1/R})$ in the $C^k$ topology, as $t\rightarrow 0^+$ (see, e.g., \cite{LWang:14}).

Note that, outside a compact set, the asymptotically conical self-expander can be parametrized by a graph over the asymptotic cone $\mC$. Namely, we have

\begin{lem} \textup{\bf (Ding \cite{Ding:20})} \label{lem:graph}
	Let $\Sigma^n \subset \R^{n+1}$ be an asymptotically conical self-expander with the asymptotic cone $\mC$. Then for $R>0$ sufficiently large, there is a function $u \in C^{\infty}(\mC\backslash B_R(0))$ so that
	$$  \graph\ u:= \{ \tilde{p}+u(\tilde{p})\nu_{\mC}(\tilde{p}) \ |\ \tilde{p}\in \mC \backslash B_R(0) \} \subset \Sigma.  $$
	Moreover, the function $u$ satisfies
	\begin{equation} \label{eq:decayofgraphfn}
		|\na^i_{\mC}u(\tilde{p})| = O(r^{-i-1}), \quad \text{for}\ i=0,1,
	\end{equation}
	as $r=|\tilde{p}|\rightarrow \infty$. Here and in the following, big O means $ O(r^{-i-1}) \leq cr^{-i-1} $ for some constant $c$.
\end{lem}

\begin{rmk}
	Actually, by the self-expander equation and an induction argument (see the proof of Lemma A.1 in \cite{Chodosh-CMS:20}), the estimate (\ref{eq:decayofgraphfn}) also holds for $i\geq 2$.
\end{rmk}

Let $\tilde{p} \in \mC\backslash B_R(0)$ for $R>0$ sufficiently large. Then there is a domain $U$ including $0$ in $\R^n$ and a local parametrization of $\mC$ in a neighborhood of $\tilde{p}$, $\tilde{F}: U \rightarrow \mC$ such that
$$\tilde{F}(0)=\tilde{p}, \quad \la \pa_i\tilde{F}(0),\pa_j\tilde{F}(0)\ra = \tilde{\delta}_{ij}, $$
and
$$\pa^2_{ij}\tilde{F}(0) = \tilde{h}_{ij}\nu_{\mC}(\tilde{p}) $$
with
$$\tilde{h}_{ij} = \la \pa^2_{ij}\tilde{F}, \nu_{\mC} \ra, \quad \tilde{h}_{ij}(0)=0 \ \text{if}\ i\neq j. $$
By Lemma \ref{lem:graph}, in a neighborhood of $ p=\tilde{p}+u(\tilde{p})\nu_{\mC}(\tilde{p})$, we can choose a local parametrization of $\Sigma$, $F: U \rightarrow \Sigma$ such that
$$ F(x) = \tilde{F}(x) + u(\tilde{F}(x))\nu_{\mC}(\tilde{F}(x)), $$
for $x\in U$. Then we have

\begin{lem} \textup{\bf (Ding \cite{Ding:20})} \label{lem:eqofsecondff}
	Let $F: U \rightarrow \Sigma$ be the parametrization as above such that $ F(x) = \tilde{F}(x) + u(\tilde{F}(x))\nu_{\mC}(\tilde{F}(x)) $, for $x\in U$. Then at the point $x=0$, we have
	$$  h_{ij} =c(x)\cdot \left \{ (\tilde{h}_{ij} -\tilde{h}_{ii}\tilde{h}_{jj}\tilde{\delta}^{ij}u + \pa_{ij}^2u ) \prod_{k}(1-\tilde{h}_{kk}u) + Q_{1ij}(x,u,\na_{\mC}u) \right \},  $$
	where $|Q_{1ij}(x,u,\na_{\mC}u)| \leq c|\tilde{p}|^{-1}|\na_{\mC}u|^2 + c|\tilde{p}|^{-2}|u||\na_{\mC}u| = O(r^{-5})$ and $c(x)=O(1)$, as $r\rightarrow \infty$.
	Also, we have 
	$$	g^{ij} = \tilde{\delta}^{ij}(1+2\tilde{h}_{ii}u) + Q_{2ij}(x,u,\na_{\mC}u), $$
	where $|Q_{2ij}(x,u,\na_{\mC}u)| \leq c|\tilde{p}|^{-2}u^2 + c|\na_{\mC}u|^2 = O(r^{-4})$, as $r\rightarrow \infty$.
\end{lem}

\begin{rmk} \label{rmk:conicalexpanders}
	For the analog results of Lemma \ref{lem:graph} and Lemma \ref{lem:eqofsecondff} for asymptotically conical self-shrinkers, see \cite{LWang:14} by L. Wang and \cite{Chodosh-S:21,Chodosh-CMS:20} by Chodosh et al. Note that in \cite{Chodosh-S:21,Chodosh-CMS:20}, the authors only stated their improved conical estimates for self-shrinkers (see Lemma 2.7 in \cite{Chodosh-S:21} and Lemma A.2 in \cite{Chodosh-CMS:20}). However, as equation (2.1) in \cite{Chodosh-S:21} also holds for self-expanders, their arguments can imply the same improved conical estimates for self-expanders.
\end{rmk}

\medskip
\section{Proof of Theorem \ref{thm:translator}}

In this section, we focus on the translator case and prove Thereom~\ref{thm:translator} as stated in the introduction. 
Let $\Sigma$ be given as in Theorem~\ref{thm:translator}. Recall that, according to the ordering of the eigenvalues in Section 2, $\kappa_1$ is the least eigenvalue. By the 2-convexity assumption, we have  $\kappa_1 + \kappa_2 > 0$, and the function ${\kappa_1}/{H}$ is bounded on $\Sigma$ by
\[\frac{-1}{n-2}\le \frac {\kappa_1}{H} < 1.\]
Suppose $\Sigma^-$ is not empty, then we have
\[\inf_{\Sigma}\left( \frac{\kappa_1}{H} \right) = \epsilon_0 <0.\]
Following Spruck and Xiao~\cite{SX:20}, we define
\begin{equation*}
	\vphi \left( r\right) :=\left\{ 
	\begin{array}{ccc}
		-r^4e^{-1/r^2} & \text{if} & r<0 \\ 
		0 & \text{if} & r\geq 0.
	\end{array}%
	\right.
\end{equation*}
It is easy to check that we have $\dot{\vphi} >0$ and $\ddot{\vphi}<0$ for $r<0$. Moreover, $\dot{\vphi}$ and $\ddot{\vphi}$ are both bounded away from 0 on $ [-r_2,- r_1]$, for any $r_2>r_1>0$. Now let 
\[\vphi := \vphi\left({\kappa_1}/{H}\right),\] and then $\vphi$ is continuous on $\Sigma$ and smooth on $\Sigma^-$.

Define the linear elliptic operator $\textit{L}$ by
\[\textit{L}\varphi=
	\mathrm{\Delta}^{\Sigma} \vphi	+ \left\langle  \nabla \varphi ,\ee_{n+1}\right\rangle
+ 2\left\langle \nabla \varphi, \frac{\nabla H}{H}\right\rangle.\]

\begin{lem}\label{lem:laplacianPhi(k/H)}
	On each connected component of $\Sigma_A^-$, we have
	\[
	\textit{L}\vphi = \ddot{\vphi} \left|\nabla\left(\frac{\kappa_1}{H}\right) \right|^2
	- 2\dot{\vphi} \sum_{l \neq 1}{\frac{\Lab{\nabla h_{l1}}^2}{H(\kappa_l-\kappa_1)}} \leq 0.
	\]
\end{lem}

\begin{proof}
	First of all, by direct computations, we have
	\begin{eqnarray} \label{eq1}
		&& \Delta^{\Sigma}\vphi\left( \frac{\kappa_1}{H} \right) \notag \\
		&=& \dot{\vphi}\Delta^{\Sigma}\lp \frac{\kappa_1}{H}\rp + \ddot{\vphi} \bigg|\na \lp\frac{\kappa_1}{H}\rp \bigg|^2  \\
		&=& \frac{\dot{\vphi}}{H^2} \lp H\Delta^{\Sigma}\kappa_1 - \kappa_1\Delta^{\Sigma}H \rp - 2\dot{\vphi} \lla  \na \lp \frac{\kappa_1}{H} \rp, \frac{\na H}{H} \rra + \ddot{\vphi} \bigg|\na \lp\frac{\kappa_1}{H}\rp \bigg|^2. \notag
	\end{eqnarray}
	On the other hand, by Lemma \ref{pre:translator} and Lemma \ref{lem:laplacian_k1}, 
	\begin{eqnarray} \label{eq2}
               H\Delta^{\Sigma}\kappa_1 - \kappa_1\Delta^{\Sigma}H 
		&=& H\lp \Delta^{\Sigma}h_{11} - 2\sum_{l \neq 1} \frac{|\na h_{l1}|^2}{\kappa_l-\kappa_1} \rp - \kappa_1\Delta^{\Sigma}H \notag \\
		&=& -\la H\na \kappa_1 - \kappa_1\na H, \ee_{n+1} \ra - 2H\sum_{l \neq 1} \frac{|\na h_{l1}|^2}{\kappa_l-\kappa_1}  \\
		&=& -H^2 \lla \na \lp \frac{\kappa_1}{H} \rp, \ee_{n+1} \rra - 2H\sum_{l \neq 1} \frac{|\na h_{l1}|^2}{\kappa_l-\kappa_1}. \notag
	\end{eqnarray}
	Combining (\ref{eq1}) and (\ref{eq2}), we get
	\begin{equation*}
		\Delta^{\Sigma} \vphi = \ddot{\vphi}\left|\na \left( \frac{\kappa_1}{H} \right) \right|^2 - 2\dot{\vphi} \lla \na \lp \frac{\kappa_1}{H} \rp, \frac{\na H}{H} \rra - \dot{\vphi} \lla \na \left( \frac{\kappa_1}{H} \right) , \ee_{n+1} \rra - 2\dot{\vphi} \sum_{l \neq 1}{\frac{\Lab{\nabla h_{l1}}^2}{H(\kappa_l-\kappa_1)}}.
	\end{equation*}

	Finally, the inequality in the lemma follows from the facts that $\dot{\vphi} >0$, $\ddot{\vphi}<0$ and $\kappa_l > \kappa_1$ for any $l\neq 1$.
\end{proof}

Now we are ready to prove Theorem ~\ref{thm:translator}. 

\medskip
\noindent {\it Proof of Theorem \ref{thm:translator}.}
We argue by contradiction. Suppose the theorem is false. Then $\Sigma^-$ is non-empty and we have
\[\xi :=\inf_{\Sigma}\varphi\left( \frac{\kappa_1}{H} \right) <0.\]

\subsection*{Case 1: interior infimum} If the infimum $\xi$ is attained at some point $p_0 \in \Sigma$, then $p_0$ must be in $\Sigma^-$. Since $\Sigma_A^-$ is open and dense in $\Sigma^-$, by Lemma \ref{lem:laplacianPhi(k/H)}, we have $L\vphi(p) \leq 0$ for any point $p \in \Sigma^-$. Hence, it follows from the strong maximum principle that  $\vphi = \vphi (\kappa_1/H)$ is a constant function on $\Sigma^-$. Then, on each connected component of $\Sigma_A^-$, we have
\begin{equation}\label{temp1}
	{\kappa_1}/{H} \equiv \varepsilon_0:=\varphi^{-1}(\xi) < 0,
\end{equation}
\begin{equation}\label{temp2}
	{\kappa_n}/{H}\ge \cdots \ge \kappa_2/H \ge -\kappa_1/H = -\varepsilon_0 > 0,
\end{equation}
\begin{equation}\label{temp3}
    \nabla \lp \frac{\kappa_1}{H} \rp = \frac{H\nabla \kappa_1-\kappa_1\nabla H}{H^2}= \frac{\nabla \kappa_1-\epsilon_0\nabla H}{H}=0.
\end{equation}
By Lemma \ref{lem:laplacianPhi(k/H)}, we get
\begin{equation}\label{temp4}
    \sum_{l \neq 1}{\frac{\Lab{\nabla h_{l1}}^2}{H(\kappa_l-\kappa_1)}}=0
\end{equation}
on each connected component of $\Sigma_A^-$ as well. 

From (\ref{temp4}), we have $\na h_{l1} = 0$ for $l\neq 1$. In particular, it follows from the Codazzi equation that $\na_l \kappa_1 =\na_l h_{11} = 0$ and $ \na_1 \kappa_l =\na_1 h_{ll} =0$ for all $l\neq 1$. Combining this with (\ref{temp3}), we get 
\[\na_1 \kappa_1 = \frac{\epsilon_0}{1-\epsilon_0} \na_1(h_{22} + \cdots + h_{nn}) = 0.\] Thus $\na \kappa_1=0$, which implies $\na H=0$ by (\ref{temp3}), i.e., $H$ is constant (on each connected component of $\Sigma_A^-$).   

However,  by Lemma~\ref{pre:translator},   $\Delta^f H + |A|^2H=0 $ hence $H = 0$. This is a contradiction to the  2-convexity assumption. Therefore, the infimum $\xi$ cannot be attained on $\Sigma$.

\subsection*{Case 2: infinity infimum} 
Since $\xi :=\inf_{\Sigma}\varphi\left( \frac{\kappa_1}{H} \right) <0$, we may take $\alpha=0$ in Lemma 2.4. Then $\vphi$ is smooth on $\{\vphi < \alpha\} = \Sigma^-$, and it follows from the Omori-Yau maximum principle (Lemma \ref{lemmaomoriyau} and Remark~\ref{OY})  that there exists  a sequence $\{p_k\}$ of points in $\Sigma$ such that
$$  \vphi(p_k) \rightarrow \xi, \quad |\na \vphi|(p_k) \rightarrow 0, \quad \Delta^{\Sigma} \vphi(p_k) \rightarrow \delta \geq 0.  $$

Since $\vphi(p_k) \rightarrow \xi<0 $, without loss of generality, we may assume all ${p_k}'s$ are in $\Sigma^-$. In view of the fact that $\Sigma_A^-$ is open and dense in $\Sigma^-$, for each $k$, we can pick some point $q_k \in \Sigma_A^-$ close to $p_k$ (after passing to the ``diagonal" subsequence) such that
\begin{equation} \label{omoriyaueq}
	\vphi(q_k) \rightarrow \xi, \quad |\na \vphi|(q_k)\equiv \dot{\vphi}|\na (\kappa_1/H)|(q_k) \rightarrow 0, \quad \Delta^{\Sigma} \vphi(q_k) \rightarrow \delta \geq 0.
\end{equation}
For the rest of the proof, the computations will be carried out around such $q_k \in \Sigma_A^-$, and all the limits are taken at  $q_k$ as $k\to \infty$.

By (\ref{omoriyaueq}) and Lemma \ref{lem:laplacianPhi(k/H)}, at $q_k$, we have
\begin{equation} \label{equationofDatq}
	\begin{split}
		\Delta^{\Sigma} \vphi &= \ddot{\vphi}\left|\na \left( \frac{\kappa_1}{H} \right) \right|^2 - 2\dot{\vphi} \lla \na \lp \frac{\kappa_1}{H} \rp, \frac{\na H}{H} \rra - \dot{\vphi} \lla \na \left( \frac{\kappa_1}{H} \right) , \ee_{n+1} \rra \\
		&\ \ \ - 2\dot{\vphi} \sum_{l \neq 1}{\frac{\Lab{\nabla h_{l1}}^2}{H(\kappa_l-\kappa_1)}} \rightarrow \delta \geq 0, \ \text {as} \  k\to \infty.
	\end{split}
\end{equation}
Since $\dot{\vphi}$ is positive and bounded away from $0$ as $k\rightarrow \infty$, it follows from (\ref{omoriyaueq}) that $\na (\kappa_1/H)(q_k) \rightarrow 0$. We also note, by the translator equation (\ref{eq:translator}), that
\begin{equation} \label{equationofnaofH}
	\nabla_i H = -h_{ij}\langle \tau_j, \ee_{n+1} \rangle,
\end{equation}
which implies $|\na H|^2 \leq |A|^2 \leq nH^2$. Hence $|\na H|/H$ is bounded. Together with the fact that  $|\ddot{\vphi}|$ is also bounded away from 0 as $k\rightarrow \infty$, it follows that the first three terms on the RSH of  (\ref{equationofDatq}) all have limits $0$ as $k\rightarrow \infty$.

Therefore,  since  $\Delta^{\Sigma} \vphi(q_k) \rightarrow \delta \geq 0$ while $-2\dot{\vphi} \sum_{l\neq 1}\frac{|\na h_{l1}|^2}{H(\kappa_l - \kappa_1)}\le 0$,  we conclude that 
\begin{equation*}
	\sum_{l\neq 1}\frac{|\na h_{l1}|^2}{H(\kappa_l - \kappa_1)} = \sum_{l\neq 1}\frac{|\na h_{l1}|^2}{H^2}\frac{H}{(\kappa_l - \kappa_1)} \rightarrow 0.
\end{equation*}
 Note that $\frac{H}{\kappa_l - \kappa_1} \geq \frac{H}{H+H} = \frac{1}{2}$ for $l\neq 1$, thus
\begin{equation} \label{eqthm1}
	\frac{|\na h_{l1}|^2}{H^2} \rightarrow 0.
\end{equation}
Combining (\ref{eqthm1}) with the Codazzi equation, we get
\begin{equation} \label{eqna1&kappa1}
	\frac{\na_l h_{11}}{H}=\frac{\na_l \kappa_1}{H} \rightarrow 0, \quad \forall\ l\neq 1, 
\end{equation}
\begin{equation} \label{eqh_{ll}}
	\frac{\na_1 h_{ll}}{H}=\frac{\na_1 \kappa_l}{H} \rightarrow 0, \quad \forall\ l\neq 1.
\end{equation}

In particular, since
\begin{equation} \label{eqnakappaH}
	\na\lp \frac{\kappa_1}{H}\rp = \frac{\na \kappa_1}{H} -\frac{\kappa_1\na H}{H^2} \rightarrow 0,
\end{equation}
 we have
\begin{equation} \label{eqnakappa1}
	\begin{split}
		\na_1\lp \frac{\kappa_1}{H}\rp &= \frac{\na_1 \kappa_1}{H} -\frac{\kappa_1\na_1H}{H^2}  \\
		&= \lp\frac{\na_1 \kappa_1}{H} -\frac{\kappa_1\na_1\kappa_1}{H^2} \rp - \frac{\kappa_1\na_1 (h_{22}+\cdots + h_{nn})}{H^2}\rightarrow 0.
	\end{split}
\end{equation}
By using (\ref{eqh_{ll}}) and the fact that $\kappa_1/H \rightarrow \epsilon_0 <0$, it follows that  the second term in the second identity of (\ref{eqnakappa1}) converges to $0$ as $k\rightarrow \infty$. So 
\begin{equation*}
	\frac{\na_1 \kappa_1}{H} -\frac{\kappa_1\na_1\kappa_1}{H^2} = \frac{\na_1 \kappa_1}{H}\lp 1-\frac{\kappa_1}{H}\rp \rightarrow 0,
\end{equation*}
hence 
\begin{equation} \label{eqna_1h_{11}}
	\frac{\na_1 \kappa_1}{H} \rightarrow 0.
\end{equation}
Combining (\ref{eqna1&kappa1}) and (\ref{eqna_1h_{11}}), we obtain
\begin{equation*}
	\frac{\na \kappa_1}{H} \rightarrow 0.
\end{equation*}
Thus, by (\ref{eqnakappaH}),
\begin{equation*}
	\frac{\na H}{H} \rightarrow 0.
\end{equation*}
Now, using (\ref{equationofnaofH}), for any $l$, we obtain
\begin{equation}\label{Ntoe_n+1}
	\frac{-\kappa_l\left\langle \tau_l , \ee_{n+1} \right\rangle}{H} =\frac{\nabla_l H}{H}\rightarrow 0.
\end{equation}
We also observe that
\begin{equation}\label{temp111}
	\lim_{k\to \infty}\frac{\kappa_n}{H}(q_k)\ge \cdots \ge \lim_{k\to \infty}\frac{\kappa_2}{H}(q_k) \ge -\lim_{k\to \infty}\frac{\kappa_1}{H}(q_k) = -\varepsilon_0 > 0.
\end{equation}
Therefore,  by (\ref{Ntoe_n+1}) and (\ref{temp111}), we conclude that the unit normal sequence $\{N(q_k)\}$  converges to $\ee_{n+1}$, and hence $H(q_k) \rightarrow 1$ as $k\to \infty$.

Now, as in Spruck-Xiao~\cite{SX:20}, we let $\Gamma_k=\Sigma-q_k$ be the hypersurface obtained from $\Sigma$ by translating $q_k$ to the origin. Since $\Sigma$ has bounded second fundamental form, so does $\Gamma_k$. By the compactness theorem for translaor (see, e.g.,~\cite{HIMW:21}), we can choose a subsequence, which we still denote by $\left\{\Gamma_k\right\}_{k = 1}^{\infty}$, such that  $\left\{\Gamma_k\right\}$ converges smoothly to some limit hypersurface $\Gamma_{\infty}$. Then $\Gamma_{\infty}$ is a 2-convex translating soliton and satisfies $H(0)=1$. Moreover, the corresponding function $\vphi$ for $\Gamma_{\infty}$ attains its minimum on $\Gamma_{\infty}$ at the origin, which is impossible by the conclusion in {\bf Case 1}. 

Therefore $\Sigma$ must be convex. This completes the proof of Theorem \ref{thm:translator}.
\hfill $\Box$

\bigskip
\section{Proof of Theorem \ref{thm:selfexpander}} 

In this section, we consider the self-expander case and prove the convexity of 2-convex asymptotically conical self-expanders with mean convex asymptotic cones. Let $\Sigma$ be given as in Theorem~\ref{thm:selfexpander} and suppose $\Sigma^-$ is not empty. By the 2-convexity assumption, we have
\[\inf_{\Sigma}\left( \frac{\kappa_1}{H} \right) = \epsilon_0 <0.\]
Following Spruck and Xiao~\cite{SX:20}, we define $\vphi (r)$ as in section 3 and let 
\[\vphi := \vphi\left({\kappa_1}/{H}\right),\] and then $\vphi$ is continuous on $\Sigma$ and smooth on $\Sigma^-$.

Define another linear elliptic operator $\bar{L}$ by
\[\bar{L}\varphi=
\mathrm{\Delta}^{\Sigma} \vphi	+ \left\langle  \nabla \varphi ,x\right\rangle
+ 2\left\langle \nabla \varphi, \frac{\nabla H}{H}\right\rangle.\]

\begin{lem}\label{lem:expanderlaplacianPhi(k/H)}
On each connected component of $\Sigma_A^-$, we have
\[
\bar{L}\vphi = \ddot{\vphi} \left|\nabla\left(\frac{\kappa_1}{H}\right) \right|^2
- 2\dot{\vphi} \sum_{l \neq 1}{\frac{|{\nabla h_{l1}}|^2}{H(\kappa_l-\kappa_1)}} \leq 0.
\]
\end{lem}

\begin{proof}
First of all, by direct computations, we have
\begin{eqnarray} \label{eq1:expander}
	&& \Delta^{\Sigma}\vphi\left( \frac{\kappa_1}{H} \right) \notag \\
	&=& \dot{\vphi}\Delta^{\Sigma}\lp \frac{\kappa_1}{H}\rp + \ddot{\vphi} \bigg|\na \lp\frac{\kappa_1}{H}\rp \bigg|^2 \\
	&=& \frac{\dot{\vphi}}{H^2} \lp H\Delta^{\Sigma}\kappa_1 - \kappa_1\Delta^{\Sigma}H \rp - 2\dot{\vphi} \lla \na \lp \frac{\kappa_1}{H} \rp, \frac{\na H}{H} \rra + \ddot{\vphi} \bigg|\na \lp\frac{\kappa_1}{H}\rp \bigg|^2. \notag
\end{eqnarray}
On the other hand, by Lemma \ref{pre:expander} and Lemma \ref{lem:laplacian_k1}, 
\begin{eqnarray} \label{eq2:expander}
	H\Delta^{\Sigma}\kappa_1 - \kappa_1\Delta^{\Sigma}H &=& H\lp \Delta^{\Sigma}h_{11} - 2\sum_{l \neq 1}{\frac{|{\nabla h_{l1}}|^2}{\kappa_l-\kappa_1}} \rp - \kappa_1\Delta^{\Sigma}H \notag \\
	&=& H\lp -\la \na h_{11}, x\ra -h_{11}(1+|A|^2) \rp - 2H\sum_{l \neq 1}{\frac{|{\nabla h_{l1}}|^2}{\kappa_l-\kappa_1}} \notag \\
	&&+ \kappa_1 \lp \la \na H, x\ra + H(1+|A|^2) \rp \\
	&=& -\la H\na \kappa_1 - \kappa_1\na H, x \ra - 2H\sum_{l \neq 1}{\frac{|{\nabla h_{l1}}|^2}{\kappa_l-\kappa_1}} \notag \\
	&=& -H^2 \lla \na \lp \frac{\kappa_1}{H} \rp, x \rra - 2H\sum_{l \neq 1}{\frac{|{\nabla h_{l1}}|^2}{\kappa_l-\kappa_1}}. \notag
\end{eqnarray}
Combining (\ref{eq1:expander}) and (\ref{eq2:expander}), we get
\begin{equation*}
	\Delta^{\Sigma} \vphi = \ddot{\vphi}\left|\na \left( \frac{\kappa_1}{H} \right) \right|^2 - 2\dot{\vphi} \lla \na \lp \frac{\kappa_1}{H} \rp, \frac{\na H}{H} \rra - \dot{\vphi} \lla \na \left( \frac{\kappa_1}{H} \right) , x \rra - 2\dot{\vphi} \sum_{l \neq 1}{\frac{|{\nabla h_{l1}}|^2}{H(\kappa_l-\kappa_1)}}.
\end{equation*}

Finally, the inequality in the lemma follows from the facts that $\dot{\vphi} >0$, $\ddot{\vphi}<0$ and $\kappa_l > \kappa_1$ for any $l\neq 1$.
\end{proof}

\begin{rmk} \label{rmk:selfshrinker}
By the same computations, the above lemma also holds similarly for self-shrinkers.
\end{rmk}

Next, using Lemma \ref{lem:graph} and Lemma \ref{lem:eqofsecondff}, we investigate the decay rate of the principal curvatures of asymptotically conical self-expanders.

\begin{lem} \label{lem:decayofprin_cur}
	Let $\Sigma^n \subset \R^{n+1}$ be an asymptotically conical self-expander with the asymptotic cone $\mC$. For $1\leq i \leq n$, denote $\kappa_i$ as the principal curvatures of $\Sigma$ and $\tilde{\kappa}_i$ as the corresponding principal curvatures of $\mC$. Then, as $r\rightarrow \infty$, we have
	$$ \kappa_i = O(r^{-3}),\ \text{if}\ \tilde{\kappa}_i = 0; \quad \kappa_i = O(r^{-1}),\ \text{if}\ \tilde{\kappa}_i \neq 0; \quad H = O(r^{-1}). $$
\end{lem}

\begin{proof}
	In this proof, without loss of generality, we may assume $i=1$ is the direction of radial $r$. First of all, we recall that for a cone $\mC$, if $j\geq 2$, we have
	$$  \tilde{g}_{11}=1,\  \tilde{g}_{1j}=0, \quad \text{and} \ \tilde{g}^{11}=1,\ \tilde{g}^{1j}=0.$$
	Moreover, if $i,j\geq 2$, we have
	$$  \tilde{g}_{ij}=O(r^2),\quad \text{and} \ \tilde{g}^{ij}=O(r^{-2}),\ \text{as}\ r\rightarrow \infty. $$
	For the second fundamental form and the mean curvature of $\mC$, as $r\rightarrow \infty$, we have
	$$  \tilde{h}_{1j}=0  \ \text{for}\ j \geq 1, \ \tilde{h}_{ij}=O(r) \ \text{for}\ i,j\geq 2,\quad \text{and} \ \tilde{H}=O(r^{-1}). $$
	
	Thus at the point $p \in \Sigma$, by Lemma \ref{lem:graph} and Lemma \ref{lem:eqofsecondff}, we get
	\begin{eqnarray*}
		h_{1j} &=& O(1) \cdot \left \{  (\tilde{h}_{1j} -\tilde{h}_{11}\tilde{h}_{jj}\tilde{\delta}^{1j}u + \pa_{1j}^2u ) \prod_{k}(1-\tilde{h}_{kk}u) + Q_{1ij}(x,u,\na_{\mC}u) \right \} \\
		&=& O(r^{-3}) \cdot O(1) + O(r^{-5}) \\
		&=& O(r^{-3}),
	\end{eqnarray*}
	and, for $i, j\geq 2$,
	\begin{eqnarray*}
		h_{ij} &=& O(1) \cdot \left \{ (\tilde{h}_{ij} -\tilde{h}_{ii}\tilde{h}_{jj}\tilde{\delta}^{ij}u + \pa_{ij}^2u ) \prod_{k}(1-\tilde{h}_{kk}u) + Q_{1ij}(x,u,\na_{\mC}u) \right \} \\
		&=& (\tilde{h}_{ij} + O(r^{-1}) +O(r^{-3}) )\cdot O(1) + O(r^{-5}) \\
		&=& \tilde{h}_{ij} \cdot O(1)+ O(r^{-1}).
	\end{eqnarray*}
	On the other hand, we have
	\begin{eqnarray*}
		g^{11} &=& \tilde{\delta}^{11}(1+2\tilde{h}_{11}u) + Q_{2ij}(x,u,\na_{\mC}u) \\
		&=& 1 + O(r^{-4}) \\
		&=& O(1),
	\end{eqnarray*}
	and, for $j \geq 2$,
	\begin{equation*}
		g^{1j} = \tilde{\delta}^{1j}(1+2\tilde{h}_{11}u) + Q_{2ij}(x,u,\na_{\mC}u) = O(r^{-4}).
	\end{equation*}
	Furthermore, for $i,j \geq 2$, we obtain
	\begin{eqnarray*}
		g^{ij} &=& \tilde{\delta}^{ij}(1+2\tilde{h}_{ii}u) + Q_{2ij}(x,u,\na_{\mC}u) \\
		&=& O(r^{-2}) \cdot O(1) + O(r^{-4}) \\
		&=& O(r^{-2}).
	\end{eqnarray*}
	
	Therefore, at the point $p \in \Sigma$, we have
	\begin{equation} \label{eq:decayofh1j}
		h^{1}_{j} = \sum_{k}g^{1k}h_{kj} = g^{11}h_{1j} + \sum_{k\geq 2}g^{1k}h_{kj} = O(r^{-3}), 
	\end{equation}	
	and, for $i,j \geq 2$,
	\begin{equation} \label{eq:decayofhij}
		h^{i}_{j} = \sum_{k} g^{ik}h_{kj} = \sum_{k}O(r^{-2})\cdot (\tilde{h}_{kj} + O(r^{-1})) = O(r^{-2})\cdot \tilde{h}_{jj} + O(r^{-3}).
	\end{equation}	
	Finally, Lemma \ref{lem:decayofprin_cur} follows from (\ref{eq:decayofh1j}) and (\ref{eq:decayofhij}).
\end{proof}

Now we are ready to prove Theorem ~\ref{thm:selfexpander}. 

\medskip
\noindent {\it Proof of Theorem \ref{thm:selfexpander}.}
Again, we argue by contradiction and assume the theorem is false. Then $\Sigma^-$ is non-empty, and we have
\[\xi :=\inf_{\Sigma}\varphi\left( \frac{\kappa_1}{H} \right) <0.\]

\subsection*{Case 1: interior infimum} If the infimum $\xi$ is attained at some point $p_0 \in \Sigma$, then $p_0$ must be in $\Sigma^-$. By Lemma \ref{lem:expanderlaplacianPhi(k/H)} and the same argument in the translator case in section 3, we conclude that the mean curvature $H$ is constant on each connected component of $\Sigma_A^-$. However,  by Lemma~\ref{pre:expander}, $\Delta^{\Sigma}H+\la \na H,x \ra = -H(1+|A|^2) $ hence $H = 0$, which is a contradiction to the  2-convexity assumption. Therefore, the infimum $\xi$ cannot be attained on $\Sigma$.

\subsection*{Case 2: infinity infimum}
Now we may assume $\vphi$ achieves its infimum at infinity. Then there exists a sequence $\{p_k\}$ of points in $\Sigma$ such that as $k\rightarrow \infty$ we have $p_k \rightarrow \infty$ and
\begin{equation} \label{eq:lamda1/H}
	\frac{\kappa_1}{H}(p_k) \rightarrow \epsilon_0 :=\varphi^{-1}(\xi) <0.
\end{equation}

We note that if $\Sigma$ is 2-convex and asymptotic to a cone $\mC$, then the sum of the smallest two principal curvatures of the asymptotic cone $\mC$ is nonnegative. On the other hand, for any cone, there is a principal curvature which is $0$. Thus the cone $\mC$ must be a convex cone, and $0$ is the smallest principal curvature of the asymptotic cone $\mC$. Then by Lemma \ref{lem:decayofprin_cur}, we have $\kappa_1=O(r^{-3})$ and $H=O(r^{-1})$, as $r\rightarrow \infty$. Moreover, as the asymptotic cone $\mC$ is mean convex so that at least one principal curvature is positive, by (\ref{eq:decayofhij}) in Lemma \ref{lem:decayofprin_cur}, the mean curvature of $\Sigma$ is exactly linear decay, i.e., $H\approx r^{-1}$, as $r\rightarrow \infty$. Thus $\kappa_1/H(p_k) \rightarrow 0$ as $p_k \rightarrow \infty$, which contradicts (\ref{eq:lamda1/H}).

It follows from {\bf Case 1} and {\bf Case 2} that $\Sigma^-$ is an empty set. Therefore $\Sigma$ must be convex. This completes the proof of Theorem \ref{thm:selfexpander}.
\hfill $\Box$

\medskip
By Remark \ref{rmk:conicalexpanders}, Remark \ref{rmk:selfshrinker}, and the same argument in the proof of Theorem \ref{thm:selfexpander}, we also obtain the following result for 2-convex asymptotically conical self-shrinkers.

\begin{thm} \label{thm:selfshrinker}
	Let $\Sigma^n \subset \R^{n+1}$, $n\geq 3$, be a complete immersed two-sided 2-convex self-shrinker asymptotic to mean convex cones. Then $\Sigma^n$ is convex.
\end{thm}

\bigskip


\begin{thebibliography}{99} 
	
	\bibitem{Angenent:91} Angenent, S.B., {\em On the formation of singularities in the curve shorting flow}, J. Differential Geom. {\bf 33} (1991), no. 3, 601--633.
		
	\bibitem{ACI:95} Angenent, S.B., Chopp, D.L., Ilamnen, T., {\em A computed example of nonuniqueness of mean curvature flow in $\R^3$}, Comm. Partial Differential Equations {\bf 20} (1995), no. 11-12, 1937--1958.
	
	
	\bibitem{BW:19} Bernstein, J., Wang, L., {\em Smooth compactness for spaces of asymptotically conical self-expanders of mean curvature flow}, Int. Math. Res. Not. IMRN 2021, no. 12, 9016--9044.
	
	\bibitem{BW:21} Bernstein, J., Wang, L., {\em The space of asymptotically conical self-expanders of mean curvature flow}, Math. Ann. {\bf 380} (2021), no. 1-2, 175--230.
	
	
	
	
	\bibitem{BL:16} Bourni, T., Langford, M., \emph{Type-{{II}} singularities of two-convex immersed mean curvature flow}, Geom. Flows, \textbf{2} (2016), no. 1, 1--17.
	
	\bibitem{BLT:20a} Bourni, T., Langford, M., Tinaglia, G., \emph{Convex ancient solutions to curve shortening flow}, Calc. Var. Partial Differential Equations {\bf 59} (2020), no. 4, Paper No. 133, 15 pp.
	
	\bibitem{BLT:20b} Bourni, T., Langford, M., Tinaglia, G.,, \emph{On the existence of translating solutions of mean curvature flow in slab regions}, Anal. PDE {\bf 13} (2020), no. 4, 1051--1072. 
	
	
	\bibitem{Chodosh-S:21} Chodosh, O., Schulze, F., {\em Uniqueness of asymptotically conical tangent flows}, Duke Math. J. {\bf 170} (2021), no. 16, 3601--3657.
	
	\bibitem{Chodosh-CMS:20} Chodosh, O., Choi, K., Mantoulidis, C., Schulze, F., {\em Mean curvature flow with generic initial data}, preprint (2020), arXiv:2003.14344.
	
	
	\bibitem{CSS:07} Clutterbuck, J., Schn{\"u}rer, O.C., Schulze, F., \emph{Stability of translating solutions to mean curvature flow}, Calc. Var. Partial Differential Equations \textbf{29} (2007), no. 3, 281--293.
	
	\bibitem{De:80} Derdzi\'nski, A., {\em Classification of certain compact Riemannian manifolds with harmonic curvature and non-parallel Ricci tensor}, Math. Z. {\bf 172} (1980), no. 3, 273--280.
	
	\bibitem{Ding:20} Ding, Q., {\em Minimal cones and self-expanding solutions for mean curvature flows}, Math. Ann. {\bf 376} (2020), no. 1-2, 359--450.
	
	\bibitem{EH:89} Ecker, K., Huisken, G., {\em Mean curvature evolution of entire graphs}, Ann. Math. (2) {\bf 130} (1989), no. 3, 453--471.
	
	\bibitem{FongM:19} Fong, F.T.-H., McGrath, P., {\em Rotational symmetry of asymptotically conical mean curvature flow self-expanders}, Commun. Anal. Geom. {\bf 27} (2019), no. 3, 599--618.
	
	\bibitem{Haslhofer:15} Haslhofer, R., \emph{Uniqueness of the bowl soliton}, Geometry \& Topology	\textbf{19} (2015), no. 4, 2393--2406.
	
	\bibitem{Helmensdorfer:12} Helmensdorfer, {\em Solitons of geometric flows and their applications}, Ph.D. thesis, University of Warwick (2012).
	
	\bibitem{HIMW:19} Hoffman, D., Ilmanen, T., Mart{\'i}n, F., White, B., \emph{Graphical translators for mean curvature flow}, Calc. Var. Partial Differential Equations \textbf{58} (2019), no. 4, Paper No. 117, 29 pp.
	
	\bibitem{HIMW:21} Hoffman, D., Ilmanen, T., Mart{\'i}n, F., White, B., \emph{Notes on translating solitons for mean curvature flow}, Minimal Surfaces: Integrable Systems and Visualisation, Springer International Publishing, Cham (2021), pp. 147--168.
	
	\bibitem{Huisken:90}	 Huisken, G., \emph{Asymptotic behavior for singularities of the mean curvature flow}, J. Differential Geom. \textbf{31} (1990), no. 1, 285--299.
	
	\bibitem{HS:99}	 Huisken, G., Sinestrari, C., \emph{Mean curvature flow singularities for mean convex surfaces}, Calc. Var. Partial Differential Equations \textbf{8} (1999), no. 1, 1--14.
	
	\bibitem{MSS:15} Mart{\'i}n, F., {Savas-Halilaj}, A., Smoczyk, K., \emph{On the topology of translating solitons of the mean curvature flow}, Calc. Var. Partial Differential Equations \textbf{54} (2015), no. 3, 2853--2882.
	
	\bibitem{Ilmanen:95} Ilmanen, T. {\em Lectures on mean curvature flow and related equations}, unpublished notes (1995).
	
	\bibitem{Rasul:10} Rasul, K., {\em Slow convergence of graphs under mean curvature flow}, Commun. Anal. Geom. {\bf 18} (2010), no. 5, 987--1008.
	
	\bibitem{Singley:75} Singley, D., \emph{Smoothness theorems for the principal curvatures and principal vectors of a hypersurface}, {Rocky Mountain Journal of Mathematics} \textbf{5} (1975), no. 1, 135--144.
	
	\bibitem{Smoczyk:21} Smoczyk, K., {\em Self-expanders of the mean curvature flow}, Vietnam J. Math. {\bf 49} (2021), no. 2, 433--445.
	
	\bibitem{SS:20} Spruck J., Sun L., {\em Convexity of 2-convex translating solitons to the mean curvature flow in $\R^{n+1}$}, J. Geom. Anal. {\bf 31} (2021), no. 4, 4074--4091.
	
	\bibitem{SX:20} Spruck J., Xiao L., {\em Complete translating solitons to the mean curvature flow in $\R^3$ with nonnegative mean curvature}, Amer. J. Math. {\bf 142} (2020), no. 3, 993--1015.
	
	\bibitem{Stavrou:98} Stavrou, N., {\em Self-similar solutions to the mean curvature flow}, J. Reine Angew. Math. {\bf 499} (1998), 189--198.
	
	\bibitem{Xin:15} Xin, Y.L., \emph{Translating solitons of the mean curvature flow}, Calc. Var. Partial Differential Equations \textbf{54} (2015), no. 2, 1995--2016.
	
	\bibitem{LWang:14} Wang, L., {\em Uniqueness of self-similar shrinkers with asymptotically conical ends}, J. Amer. Math. Soc. {\bf 27} (2014), no. 3, 613-638.
	
	\bibitem{Wang:11} Wang, X.-J., \emph{Convex solutions to the mean curvature flow}, Ann. of Math.	(2) \textbf{173} (2011), no. 3, 1185--1239.	
	
\end{thebibliography}
\end{document}